\documentclass[]{siamart251216}

\usepackage[siamart]{lindquist}

\usepackage[nocompress]{cite}
\usepackage{booktabs}
\usepackage{subcaption}
\usepackage[figuresleft]{rotating}

\usepackage{algorithmicx}
\usepackage{algpseudocode}

\usepackage{pgfplots}
\pgfplotsset{compat=1.18}

\newcommand{\matrixname}[1]{\texttt{#1}}

\definecolor{color_double}{RGB}{51,34,136} 
\definecolor{color_single}{RGB}{170,68,153} 
\definecolor{color_half}{RGB}{68,170,153} 
\definecolor{color_bfloat16}{RGB}{136,204,238} 

\title{A Mixed-Precision Model\\for Matrix-Splitting Methods%
}
\author{Neil Lindquist\thanks{NVIDIA, CA, USA 95051 (\email{nlindquist@nvidia.com}, \email{dappelhans@nvidia.com)}} \and David Appelhans\footnotemark[1]%
\and Charles W. Jackson\thanks{NASA Langley Research Center, VA USA, 23681 (\email{charles.w.jackson@nasa.gov}, \email{joseph.m.derlaga@nasa.gov})} \and Joseph M. Derlaga\footnotemark[2]}

\headers{A Mixed-Precision Model for Matrix-Splitting Methods}{Lindquist, Appelhans, Jackson and Derlaga}

\begin{document}

\maketitle

\begin{abstract}
	The performance of stationary iterative methods is memory-bound, due to the low arithmetic intensity.
	So, we propose a mixed-precision technique to reduce the amount of data movement in matrix-splitting iterative methods.
	We demonstrate both theoretically and experimentally that this technique will generally converge at a similar rate as an exact-precision iteration, depending on the precision used and the conditioning of the \(M\) matrix.
	Furthermore, we demonstrate that iterative refinement can be used to refine the solution to full double-precision accuracy.
	Using this approach, we mixed half and double precision to achieve an average speedup over uniform double precision of \(1.60\times\) for scalar and block Jacobi and \(1.25\times\) for scalar and block Gauss-Seidel.
	Furthermore, we reduced the precision in the SSOR iteration used by OVERFLOW, a NASA code for compressible fluid dynamics simulations, and achieved a \(1.14\times\) speedup to the overall application.
\end{abstract}
\begin{keywords}
	Mixed precision, reduced-precision arithmetic, Jacobi iterations, Gauss-Seidel iterations
\end{keywords}
\begin{MSCcodes}
	65F10, 65G50, 65Y04, 65Y20
\end{MSCcodes}

\section{Introduction}

Stationary iterative methods are used in many contexts to solve linear systems of equations, particularly for compressible fluid dynamics simulations~\cite{jackson2024gpu}.
However, such methods are memory-bound, due to the low arithmetic intensity.
Thus, selectively storing parts of the matrix in reduced precision will improve performance.
Specifically, we focus on solving the linear system \(Ax=b\) with the iteration
\[
	x_{i+1} = M^{-1}(b - Nx)
\]
for a splitting \(A = M + N\) for which portions of the matrix with small magnitude can be stored in reduced precision.
We especially consider storing the off-diagonal blocks in low precision since stationary iterations are most commonly used when \(N\) is smaller than \(M\).

We specifically focus on matrix-splitting methods.
Note that ``multi-stage'' stationary iterations like symmetric successive over-relaxation (SSOR) or the alternating-direction implicit method (ADI) can be formulated as a series of matrix splittings which are used on alternating iterations.
Additionally, any stationary iteration for linear equations with linear convergence induces a corresponding matrix splitting~\cite[Lemma 2.3]{lanzkron1990convergence}, although such a splitting may lack the sparsity of the original problem.
So, we only consider matrix-splitting iterations without further comment.

\subsection{Related Work}
There is a wide breadth of previous work on using mixed precision in numerical linear algebra~\cite{abdelfattah2021survey,higham2022mixed}.
Such work either stores the data in smaller formats to reduce data movement (like our work) or uses reduced precision for the computation to take advantage of the more numerous low-precision arithmetic units, depending on the arithmetic intensity of the application.
We focus on four floating-point formats that are primarily discussed for mixed precision in scientific computing, three defined in the IEEE 754-2019 specification (double, single, and half) and one originally developed by Google (BFloat16).
\Cref{tab:fp-specs} outlines the properties of these formats.
\begin{table}
	\caption{Properties of select floating-point formats.}
	\label{tab:fp-specs}
	\centering
	\begin{tabular}{rcc}
		Format & Maximum Value & Unit round-off* \\
		\midrule
		IEEE double & \(1.80\times 10^{308}\) & \(2^{-53} \approx 1.1\times 10^{-16}\) \\
		IEEE single & \(3.40\times 10^{38}\) & \(2^{-24} \approx 6.0\times 10^{-8}\) \\
		IEEE half & \(6.55 \times 10^{4}\) & \(2^{-11} \approx 4.9\times 10^{-4}\) \\
		BFloat16 & \(3.40\times 10^{38}\) & \(2^{-8} \approx 3.9\times 10^{-3}\) \\
	\end{tabular}
	
	* The maximum error for rounding to the nearest floating-point value.
\end{table}

Quite similar to our work, Walden et al.\ have previously considered storing off-diagonal blocks in low precision for the Fun3D compressible CFD code~\cite{walden2019mixed}.
Specifically, they store the diagonal blocks and right-hand side in double precision, the computed solution in single-precision, and the off-diagonal blocks in either single- or half-precision.
Additionally, they developed an approach to convert the data to low precision without requiring a second storage array.
However, their explorations were limited to the linear systems arising in an unstructured supersonic fluid simulation.
Furthermore, they did not provide any theoretical analysis of the error or consider the use of iterative refinement.

Anzt et al.\ have explored a mixed-precision Jacobi iteration where a different precision is used for each row~\cite{anzt2015adaptive,grutzmacher2020customized}.
Their approach is based on determining what precision to use for each row as the iteration progresses.
This allows the solver to adapt the precision both temporally and row-wise.
However, this approach leaves parts of the matrix in high-precision and introduces additional overheads for dispatching between precisions.
Originally, the technique was proposed for FPGAs~\cite{anzt2015adaptive}, but the development of the split-mantissa storage format made it applicable for general-purpose hardware~\cite{grutzmacher2020customized}.
A related work uses different precisions for the diagonal blocks of a block-Jacobi preconditioner~\cite{flegar2021adaptive}.

Another related work is the mixed-precision work of Ge and Zhang for the general alternating-direction implicit framework~\cite{ge2025three}.
They found that using an alternating-direction sweep as the inner solver of iterative refinement allows a full-precision solution to converge, even with IEEE half-precision.
However, they found it necessary to introduce a regularization term that reduced the size of the solution updates; this regularization required tuning for the specific matrix.
Additionally, the regularization worked by reducing the size of the solution update, which can increase the number of iterations by requiring the same error-direction to be corrected multiple times.
Finally, they recompute the residual every iteration, reducing the potential speedup.

The work of Anzt et al.\ also emphasizes decoupling the storage format from the data type used to do the computation in the so-called Memory Accessor paradigm~\cite{anzt2015adaptive,grutzmacher2020customized}.
For memory-bound workloads, the compute resources are otherwise largely idle, so extra computation to decode the storage format or doing higher-precision arithmetic will not affect performance.
Furthermore, decoupling the storage and compute types allows the use of formats other than IEEE floating-point, which can better exploit the application's specific data.
This has recently been expanded by Amestoy et al.\ into a memory accessor that decodes an entire block, keeping the compute datatype in the L1 cache~\cite{amestoy2025blas}.

Finally, similar to our results, previous work has shown that reducing the precision of a variety of iterative methods usually has little effect on the rate of convergence for sufficiently well-behaved problems~\cite{lindquist2020improving,lindquist2022accelerating,kelly2022newtons,mccormick2021algebraic}.
(The meaning of ``well-behaved'' naturally varies with the method in question, but usually means sufficiently well conditioned.)
However, reduced precision almost always reduces the maximum achievable accuracy;
in such cases, high-precision iterative refinement can often be used to produce a full-precision solution.
Generally, this can be understood by noting that during the initial iterations, the mathematical errors (i.e., those present in exact arithmetic) are several orders of magnitude larger than the round-off errors, and the latter is a minuscule percent of the total error and unnoticeable.\footnote{This is similar to Amdahl's law for improving software performance.}
Then, as the mathematical error is reduced to a similar or smaller magnitude as the round-off errors, the round-off error is no longer a negligible term and soon dominates, causing the accuracy to stagnate at the working-precision accuracy.

\section{Theoretical Analysis}
\label{sec:theory}

To understand this approach, we bound the normwise forward error of a generic stationary iteration.
In short, the theorem shows that the mixed-precision approach will converge almost as fast as an exact iteration until it reaches a relative forward error determined by the storage and computation accuracy.
Note that while reducing the precision of off-diagonal blocks is our focus, the theorem applies to a fairly general class of mixed-precision methods.

\begin{theorem}
	\label{thm:main-result}
	Let \(A\) be decomposed as \(M + N\), and let \(\widehat{M} = M - E_M\) and \(\widehat{N} = N - E_N\) be the values that are actually stored.
	Assume \(\norm{\infty}{E_M} \leq \epsilon_M \norm{\infty}{M}\) and \(\norm{\infty}{E_N} \leq \epsilon_N \norm{\infty}{N}\).
	Assume \(\norm{\infty}{M^{-1}E_M} \leq 0.5\).
	Let \(n_z \geq 1\) be the largest number of nonzeros in any row of \(N\).
	Suppose that all computation uses floating point arithmetic with a unit round-off of \(\epsilon_c\) and that the solution to \(\widehat{M}x_{i+1} = f_i\) is computed so that
	\[
	(\widehat{M} + E_{S})\widetilde{x}_{i+1} = \widetilde{f}_i
	\]
	where \(\norm{\infty}{E_S} \leq c_n\epsilon_c\norm{\infty}{\widehat{M}}\).
	Finally, assume that \(\norm{\infty}{M^{-1}N} < 1\), i.e., the iterations converge in exact arithmetic.
	Then,
	\begin{align*}
		\norm{\infty}{\widetilde{e}_{i+1}}
		\leq&
		(1+2\norm{\infty}{M^{-1}E_M})\norm{\infty}{M^{-1}N}\norm{\infty}{\widetilde{e}_{i}} \\
		&+ \norm{\infty}{\widehat{M}^{-1}}\norm{\infty}{N}((\epsilon_N+\gamma_{n_z})\norm{\infty}{\widetilde{x}_i} + \epsilon_c\norm{\infty}{x}) \\
		&+ (\epsilon_M + \epsilon_c)\cond{\infty}{\widehat{M}}\norm{\infty}{x} \\
		&+ c_n\epsilon_c\cond{\infty}{\widehat{M}}\norm{\infty}{\widetilde{x}_{i+1}}
	\end{align*}
	where \(\gamma_{n_z} = \epsilon_cn_z / (1 - \epsilon_cn_z)\)~\cite[Chap.~3]{higham2002accuracy}.
	Furthermore, if we have \(\norm{\infty}{\widetilde{x}_i} \lesssim \norm{\infty}{x}\), \({\norm{\infty}{\widetilde{x}_{i+1}} \lesssim \norm{\infty}{x}}\), \(\epsilon_N \leq 0.1\), and \(n_z\epsilon_c \ll 1\), then
	\begin{align*}
		\norm{\infty}{\widetilde{e}_{i+1}}
		\lesssim&
		(1+2\norm{\infty}{M^{-1}E_M})\norm{\infty}{M^{-1}N}\norm{\infty}{\widetilde{e}_{i}} \\
		&+ (\epsilon_N + 2.1n_z\epsilon_c)\norm{\infty}{\widehat{M}^{-1}}\norm{\infty}{N}\norm{\infty}{x}\\
		&+ (\epsilon_M + (c_n+1)\epsilon_c)\cond{\infty}{\widehat{M}}\norm{\infty}{x}.
	\end{align*}
\end{theorem}
\begin{proof}
	We begin by bounding some intermediate quantities.
	By the Woodbury formula,
	\[
		\widehat{M}^{-1} = (I + M^{-1}E_M(I - M^{-1}E_M)^{-1})M^{-1}.
	\]
	By the assumption on \(\norm{\infty}{M^{-1}E_M}\), we have
	\begin{equation}
		\label{eq:M-perturb-bound}
		\norm{\infty}{I + M^{-1}E_M(I - M^{-1}E_M)^{-1}} \leq 1 + 2\norm{\infty}{M^{-1}E_M} \leq 2.
	\end{equation}
	Thus,
	\begin{equation}
		\label{eq:M-hat-bound}
		\norm{\infty}{\widehat{M}} \leq (1+\epsilon_M)\norm{\infty}{M}
		\quad\text{and}\quad
		\norm{\infty}{\widehat{M}^{-1}} \leq (1+2\norm{\infty}{M^{-1}E_M})\norm{\infty}{M^{-1}}.
	\end{equation}
	Next, we consider the error introduced by the round-off errors in the computation.
	By basic error analysis of matrix-vector products~\cite[(3.12)]{higham2002accuracy}, \(b - \widehat{N}\widetilde{x}_i\) is computed such that
	\begin{equation}
		\label{eq:matvec-err}
		\widetilde{f}_i = b - \widehat{N}\widetilde{x}_i + E_{fi}
		\qquad \text{and} \qquad
		\norm{\infty}{E_{fi}} \leq \epsilon_c \norm{\infty}{b} + \gamma_{n_z}\norm{\infty}{\widehat{N}}\norm{\infty}{\widetilde{x}_i}.
	\end{equation}
	Combined with the assumption on the application of the inversion of \(\widehat{M}\), the actual iteration is
	\[
	(\widehat{M} + E_S)\widetilde{x}_{i+1} = b - \widehat{N}\widetilde{x}_i + E_{fi},
	\]
	which is equivalent to
	\begin{equation}
		\label{eq:redux:iter-w/-compute-err}
		\widetilde{x}_{i+1} = \widehat{M}^{-1}(b - \widehat{N}\widetilde{x}_i + E_{fi} - E_S\widetilde{x}_{i+1}).
	\end{equation}
	
	Finally, we can bound \(\widetilde{e}_{i+1} = \widetilde{x}_{i+1} - x\) in terms of \(\widetilde{e}_i\).
	Note that \(Ax = b\) can be rearranged to
	\begin{equation}
		\label{eq:redux:iter-fixed-point}
		x = \widehat{M}^{-1}(b - (N + E_M)x).
	\end{equation}
	Then, substituting \cref{eq:redux:iter-w/-compute-err,eq:redux:iter-fixed-point} into the error definition gives
	\[
		\widetilde{e}_{i+1} = \widehat{M}^{-1}(b - \widehat{N}\widetilde{x}_i + E_{fi} - E_S\widetilde{x}_{i+1}) - \widehat{M}^{-1}(b - (N + E_M)x).
	\]
	Rearranging equations gives
	\begin{equation}
		\label{eq:simple-e}
		\widetilde{e}_{i+1} = \widehat{M}^{-1}N\widetilde{e}_{i} + \widehat{M}^{-1}(E_N\widetilde{x}_i + E_Mx + E_{fi} - E_S\widetilde{x}_{i+1}).
	\end{equation}
	Hence,
	\begin{align*}
		\norm{\infty}{\widetilde{e}_{i+1}}
		\leq&
		(1+2\norm{\infty}{M^{-1}E_M})\norm{\infty}{M^{-1}N}\norm{\infty}{\widetilde{e}_{i}} \\
		&+ \epsilon_N\norm{\infty}{\widehat{M}^{-1}}\norm{\infty}{N}\norm{\infty}{\widetilde{x}_i} + \epsilon_M\cond{\infty}{\widehat{M}}\norm{\infty}{x}
		+ \norm{\infty}{\widehat{M}^{-1}}\norm{\infty}{E_S}\norm{\infty}{\widetilde{x}_{i+1}} \\
		&+ \norm{\infty}{\widehat{M}^{-1}}(\epsilon_c\norm{\infty}{b} + (1+\epsilon_N)\gamma_{n_z}\norm{\infty}{N}\norm{\infty}{\widetilde{x}_i}) \\
		\leq&
		(1+2\norm{\infty}{M^{-1}E_M})\norm{\infty}{M^{-1}N}\norm{\infty}{\widetilde{e}_{i}} \\
		&+ \epsilon_N\norm{\infty}{\widehat{M}^{-1}}\norm{\infty}{N}\norm{\infty}{\widetilde{x}_i} + \epsilon_M\cond{\infty}{\widehat{M}}\norm{\infty}{x} \\
		&+ \epsilon_c\cond{\infty}{\widehat{M}}\norm{\infty}{x} + \norm{\infty}{\widehat{M}^{-1}}\norm{\infty}{N}(\epsilon_c\norm{\infty}{x} + (1+\epsilon_N)\gamma_{n_z}\norm{\infty}{\widetilde{x}_i}) \\
		&+ c_n\epsilon_c\cond{\infty}{\widehat{M}}\norm{\infty}{\widetilde{x}_{i+1}}.
	\end{align*}
	
	Finally, assume \(\norm{\infty}{\widetilde{x}_i} \lesssim \norm{\infty}{x}\), \(\norm{\infty}{\widetilde{x}_{i+1}} \lesssim \norm{\infty}{x}\), \(\epsilon_N \leq 0.1\), and \(n_z\epsilon_c \ll 1\).
	Therefore,
	\begin{align*}
		\norm{\infty}{\widetilde{e}_{i+1}}
		\lesssim&
		(1+2\norm{\infty}{M^{-1}E_M})\norm{\infty}{M^{-1}N}\norm{\infty}{\widetilde{e}_{i}} \\
		&+ (\epsilon_N + 2.1n_z\epsilon_c)\norm{\infty}{\widehat{M}^{-1}}\norm{\infty}{N}\norm{\infty}{x} + (\epsilon_M + \epsilon_c)\cond{\infty}{\widehat{M}}\norm{\infty}{x} \\
		&+ c_n\epsilon_c\cond{\infty}{\widehat{M}}\norm{\infty}{x}.
	\end{align*}
\end{proof}

Assuming \(\norm{\infty}{\widetilde{x}_{i+1}} \lesssim \norm{\infty}{x}\) to bound \(\norm{\infty}{\widetilde{e}_{i+1}}\) may appear to be a circular argument.
However, note that rearranging \cref{eq:simple-e} gives
\begin{align*}
	(I - \widehat{M}^{-1}E_S)^{-1}\widetilde{x}_{i+1} = \widehat{M}^{-1}N\widetilde{e}_{i} + \widehat{M}^{-1}(E_N\widetilde{x}_i + E_Mx + E_{fi}).
\end{align*}
So if
\begin{align*}
	\norm{\infty}{\widetilde{e}_i} \lesssim \norm{\infty}{\widetilde{x}_i} &\lesssim \norm{\infty}{x}, \\
	c_n\epsilon_c\cond{\infty}{\widehat{M}} &\ll 1, &\text{and} \\
	(\epsilon_c + \epsilon_M + \epsilon_N)(\cond{\infty}{\widehat{M}} + \norm{\infty}{\widehat{M}^{-1}}\norm{\infty}{N}) &\ll 1,
\end{align*}
then \(\norm{\infty}{\widetilde{x}_{i+1}} \lesssim \norm{\infty}{x}\).

Additionally, there are a few specific quantities worth considering.
First, note that if \(\epsilon_M\cond{\infty}{\widehat{M}} \leq 0.5\), then \(\norm{\infty}{M^{-1}E_M} \leq 0.5\).
Furthermore, \cref{eq:M-hat-bound} implies that bounds relative to \(\norm{\infty}{\widehat{M}^{-1}}\) and \(\cond{\infty}{\widehat{M}}\) can be easily replaced with bounds based on \(M\).
Finally, probabilistic analysis says that the \(\gamma_{n_z}\) introduced in \cref{eq:matvec-err} will scale as \(\sqrt{n_z}\epsilon_c + O(\epsilon_c^2)\) instead of \(n_z\epsilon_c + O(\epsilon_c^2)\) for large \(n_z\) with high probability~\cite{higham2019new}.

\Cref{thm:main-result} implies that for sufficiently small \(\epsilon_M\) and \(\epsilon_N\), the reduced-precision iteration will converge almost as fast as exact arithmetic until it reaches a relative accuracy dependent on \(\epsilon_c\), \(\epsilon_M\), \(\epsilon_N\) \(\cond{\infty}{M}\), and \(\norm{\infty}{M^{-1}}\norm{\infty}{N}\).

This bound additionally provides insight into the selection of \(\epsilon_M\), \(\epsilon_N\), and \(\epsilon_c\).
As is usually the case, there is no numerical benefit in storing \(M\) and \(N\) in higher precision that the computation is done in.
Second, the approximation of \(N\) doesn't actually affect the rate of convergence but only the maximum achievable accuracy.
Thus, \(\epsilon_N\) can be chosen such that \(\epsilon_N\norm{\infty}{\widehat{M}^{-1}}\norm{\infty}{N}\) is a little less that the desired relative forward error.
Furthermore, as stationary iterations are generally bandwidth bound, \(c_n\) should ideally be chosen so that
\[
	\epsilon_N \geq (n_z+1)\epsilon_c \qquad\text{and}\qquad \epsilon_M \geq (c_n+1)\epsilon_c.
\]
If such a precision is not supported natively in hardware, it is reasonable to decrease the accuracy of \(N\) so that \(\epsilon_N\) is somewhere between \((n_z+1)\epsilon_c\) and \(\sqrt{n_z}\epsilon_c\).
The hardest value to develop a heuristic for is \(\epsilon_M\) as this value directly affects the rate of convergence.
\Cref{thm:main-result} assumes that \(\epsilon_M\cond{\infty}{M} \leq 0.5\) and convergence requires \((1 + 2\epsilon_M\cond{\infty}{M})\norm{\infty}{M^{-1}N} \ll 1\), so the choice of \(\epsilon_M\) should reflect that.
Finally, similar analysis to that of \(\epsilon_N\) should be applied to \(\epsilon_M\) to ensure the maximum achievable accuracy is sufficient.

A key part of the bounds in \cref{thm:main-result} is the condition number \(\cond{\infty}{M}\) and the pseudo-condition number \(\norm{\infty}{M^{-1}}\norm{\infty}{N}\).
Note that if \(D\) is a diagonal matrix that can be applied without round-off error (i.e., it's elements are powers of two), then the rounding errors incurred by the stationary iterations
\[
	DMx_{i+1} = (Db - DNx_i)
	\qquad\text{and}\qquad
	Mx_{i+1} = (b - Nx_i)
\]
will be identical.
However, the choice of \(D\) affects \(\norm{\infty}{(DM)^{-1}}\norm{\infty}{DN}\) and \(\cond{\infty}{DM}\).
This results in the following theorem, which states that the convergence isn't affected by the scaling of block rows.
An analogous theorem can also be proven for column-scaling with the 1-norm.

\begin{corollary}
	\label{thm:block-row-scaling}
	Assume \(A\) is decomposed into block rows \(\mathcal{I}_1, \mathcal{I}_2, \dots, \mathcal{I}_b\).
	Let \(\epsilon_M\), \(\epsilon_N\), and \(c_n\) be such that for every \(1 \leq i \leq b\),
	\[
		\norm{\infty}{E_M[\mathcal{I}_i, :]} \leq \epsilon_M\norm{\infty}{M[\mathcal{I}_i, :]},
		\qquad
		\norm{\infty}{E_N[\mathcal{I}_i, :]} \leq \epsilon_N\norm{\infty}{N[\mathcal{I}_i, :]},
	\]
	\[
		\text{and}\qquad \norm{\infty}{E_S[\mathcal{I}_i, :]}\leq c_n\epsilon_c\norm{\infty}{M[\mathcal{I}_i, :]}
	\]
	(i.e., the errors of a block row are relative to that row).
	Let \(D\) be any diagonal matrix containing powers of two such that the diagonal values of a given block row are identical,\footnotemark{} and assume that \(\cond{\infty}{DM}\epsilon_M \leq 0.5\).
	\footnotetext{I.e., for \(1\leq i\leq b\), the submatrix \(D[\mathcal{I}_i, \mathcal{I}_i]\) is a power of two times the identity matrix.}
	Then,
	\[
		\norm{\infty}{M^{-1}E_M} \leq \epsilon_M\cond{\infty}{DM}.
	\]
	Hence,
	\begin{align*}
		\norm{\infty}{\widetilde{e}_{i+1}}
		\leq&
		(1+2\epsilon_M\cond{\infty}{DM})\norm{\infty}{M^{-1}N}\norm{\infty}{\widetilde{e}_{i}} \\
		&+ \norm{\infty}{(D\widehat{M})^{-1}}\norm{\infty}{DN}((\epsilon_N+\gamma_{n_z})\norm{\infty}{\widetilde{x}_i} + \epsilon_c\norm{\infty}{x}) \\
		&+ (\epsilon_M + \epsilon_c)\cond{\infty}{D\widehat{M}}\norm{\infty}{x} \\
		&+ c_n\epsilon_c\cond{\infty}{D\widehat{M}}\norm{\infty}{\widetilde{x}_{i+1}}.
	\end{align*}
	and if \(\norm{\infty}{\widetilde{x}_i} \lesssim \norm{\infty}{x}\), \(\norm{\infty}{\widetilde{x}_{i+1}} \lesssim \norm{\infty}{x}\), and \(n_z\epsilon_c \ll 1\), then
	\begin{align*}
		\norm{\infty}{\widetilde{e}_{i+1}}
		\lesssim&
		(1+2\epsilon_M\cond{\infty}{DM})\norm{\infty}{M^{-1}N}\norm{\infty}{\widetilde{e}_{i}} \\
		&+ (\epsilon_N + (n_z+1)\epsilon_c)\norm{\infty}{(D\widehat{M})^{-1}}\norm{\infty}{DN}\norm{\infty}{x}\\
		&+ (\epsilon_M + (c_n+1)\epsilon_c)\cond{\infty}{D\widehat{M}}\norm{\infty}{x}.
	\end{align*}
	Because the error is bounded by the smallest such \(D\), poor row scaling is not detrimental to the convergence and accuracy.
\end{corollary}
\begin{proof}
	By the choice of \(\epsilon_M\), \(\epsilon_N\), and \(c_n\), we have
	\[
		\norm{\infty}{DE_M} \leq \epsilon_M\norm{\infty}{DM},
		\qquad
		\norm{\infty}{DE_N} \leq \epsilon_N\norm{\infty}{DN},
	\]
	\[
		\text{and}\qquad \norm{\infty}{DE_S}\leq c_n\epsilon_c\norm{\infty}{DM}.
	\]
	Then, \cref{thm:main-result} applies and provides the result.
\end{proof}

It may not be obvious that the block-row assumption on \(E_S\) holds for the usual stationary methods.
So, consider the case where \(M\) is block lower triangular (which includes the block Jacobi, Gauss-Seidel, and successive over-relaxation iterations).
Then, \(E_S\) is the error incurred in computing
\[
	M[\mathcal{I}_i,\mathcal{I}_i]^{-1}\Biggl(\widetilde{f}_i[\mathcal{I}_i] - \sum_{j=1}^{i-1} M[\mathcal{I}_i,\mathcal{I}_j]\widetilde{x}_{i+1}[\mathcal{I}_j] \Biggr).
\]
By standard error results~\cite[(3.12)]{higham2002accuracy}, the error in computing the parenthetical quantity is bounded by \({\gamma_{n}|M[\mathcal{I}_i, :]||\widetilde{x}_{i+1}|}\).
The error on inverting the diagonal block depends on the specific algorithm, but generally Gaussian elimination is used.
The error for Gaussian elimination depends on a small polynomial of the block size times the growth factor.
For many important classes of matrices (e.g., symmetric positive definite or diagonally dominant), the growth factor is a small constant.
For general matrices, the growth factor can be much larger (especially if pivoting is not used), but if the diagonal blocks all have similar enough numerical structure, the growth factor won't vary significantly between them.
So, there will generally be a constant \(c_n\) that encompasses the polynomial and growth factor terms.

\subsection{Higher accuracy with iterative refinement}
\label{sec:theory:ir}

A key limitation of the above bounds is the maximum achievable accuracy.
While some applications may only need a few digits of accuracy, many others require higher accuracy.
The standard technique to produce a high-precision solution from a low-accuracy solver is mixed-precision iterative refinement.
\Cref{alg:restarted-stationary} describes the mixed-precision algorithm.
\begin{algorithm}[tbp]
	\caption{Restarted Stationary Iterations}
	\label{alg:restarted-stationary}
	\begin{algorithmic}[1]
		\State \(x_0 = 0\)
		\For{\(i\) in \(0, \dots, \ell\)}
		\State \(r_i = b - Ax_i\)
		\State \(y_{i,0} = 0\)
		\For{\(j\) in \(0, \dots, k_i\)}
		\State \(y_{i,j+1} = M^{-1}r_i - M^{-1}Ny_{i,j}\)
		\EndFor
		\State \(x_{i+1} = x_{i} + y_{i,k_i+1}\)
		\EndFor
	\end{algorithmic}
\end{algorithm}
Note that due to the similarity with restarted GMRES, we refer to the outer iterations as restarts.
Furthermore, recall that iterative refinement is equivalent to Newton's method~\cite[Sec.~2.5]{demmel1997applied},
so applications where the stationary iteration is embedded within Newton's method may be already able to refine the solution to higher accuracy.

Usefully, we can show that in exact arithmetic using \cref{alg:restarted-stationary} is equivalent to just doing the stationary iteration.
Consider
\[
	z_{i,j} = x_{i} + y_{i,j}.
\]
Note that \(z_{i,k_i+1} = z_{i+1,0}\) and
\begin{align*}
	z_{i,j+1}
	&= x_i + (M^{-1}r_i - M^{-1}Ny_{i,j}) \\
	&= x_i + M^{-1}b - x_i - M^{-1}Nx_i - M^{-1}Ny_{i,j} \\
	&= M^{-1}b - M^{-1}Nz_{i,j}.
\end{align*}
Thus, \(z_{i,j}\) are the iterates of the stationary iteration without iterative refinement.\footnote{Similarly, taking \(k_i=0\) also shows that a stationary iteration is equivalent in exact precision to iterative refinement with \(M^{-1}\) as the inner solver.}
This is a notable result as it implies that restarting will not affect the rate of convergence in exact arithmetic.
This stands in sharp contrast to restarted GMRES for which the rate of convergence can be significantly affected by the choice of restarts~\cite{vandervorst1993superlinear,embree2003tortoise}.
The difference is that more iterations before restarting gives GMRES a larger subspace to optimize over~\cite{vandervorst1993superlinear}, while each iteration in a stationary method is independent of any other.

Within the iterative refinement setting, the inner iteration can either run until low precision is saturated or it can restart sooner.
If it restarts sooner, then \cref{thm:main-result} and the equivalence with non-restarted iterations indicate that the mixed-precision approach will converge similarly to a full-precision implementation.
So, consider the case where low-precision arithmetic is saturated so that \(\norm{\infty}{A^{-1}r_i - y_{i,k_i+1}} \leq c\epsilon \norm{\infty}{A^{-1}r_i}\).
With a slight modification to a bound by Carson and Higham\footnotemark~\cite[Thm.~3.2 and Cor.~3.3]{carson2018accelerating},
\footnotetext{Their formulation of iterative refinement includes an extra rounding of the residual.}
it is straightforward to show that if \(c\epsilon \ll 1\), the residual will improve by a factor of about \(c\epsilon\) after each restart until an \(\hat{x}\) is computed which satisfies
\[
	\frac{\norm{\infty}{x - \hat{x}}}{\norm{\infty}{x}}
	\lesssim 4n_z\epsilon_c\mathrm{cond}(A, x)
\]
where \(\epsilon_c\) is the unit round-off for the arithmetic used to compute the residual and
\[
	\mathrm{cond}(A, x) = \frac{\norm{\infty}{|A|\,|A^{-1}|\,|x|}}{\norm{\infty}{|x|}} \leq \cond{\infty}{A}.
\]

An alternative to iterative refinement is to increase the accuracy of the computation as the iteration progresses similar to the work of Anzt et al.~\cite{anzt2015adaptive,grutzmacher2020customized}.
However, in such an approach, only the early iterations benefit from reduced precision.
Thus, we expect iterative refinement to provide better performance if moderate to high accuracy is required.

\section{Experimental Results}

To demonstrate the practical effectiveness of this mixed-precision approach, we tested it on various matrices from the SuiteSparse collection~\cite{davis2011university} and NASA's OVERFLOW code for compressible computational fluid dynamics~\cite{jackson2024gpu}.
While the theoretical analysis allows for variable block sizes, our implementation assumes a uniform block size for each matrix.

Both test codes are based on using stationary iterations to solve the equation \((D^{-1}A)x = (D^{-1}b)\) where \(D\) is the block diagonal of \(A\).
This has several major benefits over directly iterating on \(A\).
First, it normalizes the scale of the rows, preventing problematic overflow and underflow in the limited dynamic range of half precision.
Otherwise, values larger than \num{65504} will overflow to infinity and values smaller than \num{6.10e-5} will lose precision.
Second, this preprocessing step avoids the cost of doing the triangular solves for the block-rows in each iteration, which improves performance.
Finally, at the start of the solve or after a restart, we can take the initial solution vector to be the right-hand side; this is equivalent to a Jacobi iteration, at the low cost of a vector copy.
Because of this and the cost of computing the residual, a restart is counted as an iteration.

Our data is available in the supplemental materials.\footnotemark
\footnotetext{Note that, due to NASA restrictions, the source code for the half-precision OVERFLOW is not available as part of our code artifact.  However, the half-precision implementation is available to US Persons in version 2.5.2 of OVERFLOW (\url{https://software.nasa.gov/software/LAR-20382-1}).}
The tests were run on a NVIDIA H200 SXM GPU with a \SI{140}{\gibi\byte} memory with a peak bandwidth of \SI{4.8}{\tebi\byte\per\second}, and 132 SMs with a peak single-precision throughput of \SI{67}{\tera\flop\per\second}.
The code was compiled and run in the 26.5 NVIDIA HPC SDK Docker container\footnote{\url{nvcr.io/nvidia/nvhpc:26.5-devel-cuda_multi-ubuntu24.04}} with CUDA Toolkit 13.2 and Ubuntu 24.04.

\subsection{SuiteSparse matrices}

We begin by testing the mixed-precision approach with block Jacobi and block Gauss-Seidel iterations on a variety of matrices from the SuiteSparse collection~\cite{davis2011university}.
\Cref{tab:test-matrices} shows the specific matrices we used.
\begin{table}
	\caption{Test matrices from the SuiteSparse collection}
	\label{tab:test-matrices}
	\centering
	\begin{tabular}{lS[table-format=8.0]S[table-format=9.0]rc}
		Name & \(n\) & \(n_{nz}\) & \(n_b\) & Kind \\
		\midrule
		\matrixname{atmosmodj} & 1270432 & 8814880 & 4 & Computational Fluid Dynamics \\
		\matrixname{cage14} & 1505785 & 27130349 & 5 & Directed Weighted Graph \\
		\matrixname{Dubcova2} & 65025 & 1030225 & 1 & 2D/3D Problem \\
		\matrixname{language} & 399130 & 1216334 & 5 & Directed Weighted Graph \\
		\matrixname{poisson3Db} & 85623 & 2374949 & 1 & Computational Fluid Dynamics \\
		\matrixname{thermomech\_dM} & 204316 & 1423116 & 4 & Thermal Problem \\
		\matrixname{torso2} & 115967 & 1033473 & 1 & 2D/3D Problem \\
	\end{tabular}
\end{table}
(Note that some matrices use a block size of 1, making them equivalent to scalar Jacobi or Gauss-Seidel.)
These matrices were chosen based on whether the double-precision solvers can converge within a reasonable number of iterations.
The right-hand sides were generated from a known solution vector of \(x[i] = \sin(2i\pi/n)\) (\(i = 0, \dots, n-1\)).
This allows our analysis to include the forward error.

We begin by looking at the rate of convergence of block Jacobi on representative matrices in the infinity-norm forward error,
\[
	\frac{\norm{\infty}{x - x^*}}{\norm{\infty}{x^*}},
\]
shown in \Cref{fig:convergence:jac-cage14,fig:convergence:jac-language}.
\begin{figure}
	\centering
	\input{results/latex-plot-acc-jac-cage14.tex}
	\caption{Convergence rate of block Jacobi on \matrixname{cage14} with different precisions.  Vertical lines indicate the iterations in which restarts occur.}
	\label{fig:convergence:jac-cage14}
\end{figure}
\begin{figure}
	\centering
	\input{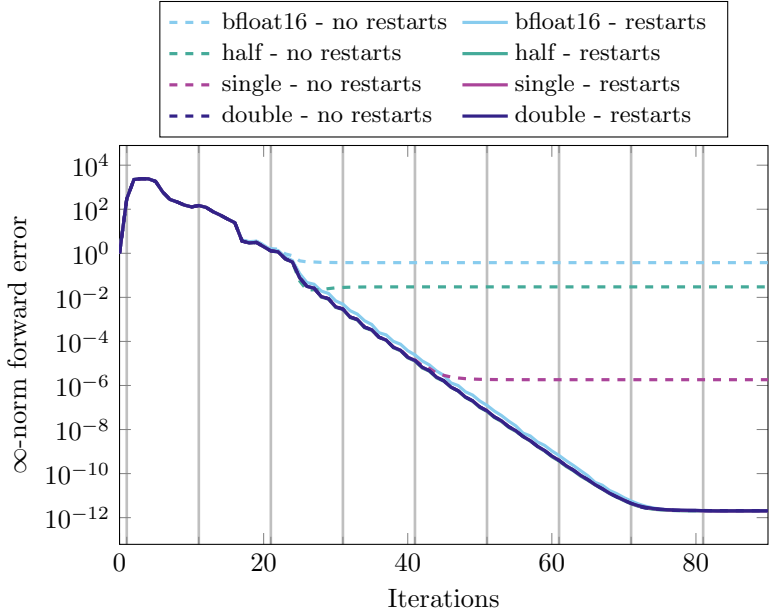}
	\caption{Convergence rate of block Jacobi on \matrixname{language} with different precisions.  Vertical lines indicate the iterations in which restarts occur.}
	\label{fig:convergence:jac-language}
\end{figure}
There are several key observations.

First, without restarting, the mixed-precision implementations have the same initial rate of convergence as the double-precision implementation.
This aligns with the theory in \cref{sec:theory} and previous work on reduced-precision iterative methods~\cite{lindquist2020improving,lindquist2022accelerating,kelly2022newtons}.
The ability for low-precision arithmetic to converge just as quickly as high-precision arithmetic is quite powerful.
In some contexts (such as Newton's method), we only need a solution accurate to a few decimal places, thus rate of convergence is the only important aspect.

Furthermore, even with restarting, the mixed-precision convergence behaves reasonably close to the double precision, non-restarted convergence and can converge to full double-precision accuracy.
This comes from the fact, as discussed in \cref{sec:theory:ir}, that the restarted iteration is equivalent to the non-restarted iteration.
So, for sufficiently frequent restarts, the precision error will remain a relatively small component of the iterate and have limited effect on the rate of convergence.
In the extreme case, restarting every iteration is equivalent to doing the solve entirely in double precision.

Note that Jacobi does not reach a forward error of \(10^{-15}\) on the \matrixname{language} matrix for any precision.
Both \cref{thm:main-result} and previous results~\cite[pg.~326]{higham2019new} suggest that an error will generally be greater than \(u\norm{\infty}{M^{-1}}\norm{\infty}{N}\).
For this Jacobi iteration matrix,\footnotemark{} \(\norm{\infty}{M^{-1}}\norm{\infty}{N} = \num{1.623e3}\). 
\footnotetext{Although, the spectral radius of the iteration matrix is approximately \(\num{0.6}\) which aligns with the latter part of the convergence.  See \texttt{analysis.jl} in our code artifact for more details.}
So, a maximum achievable accuracy on the order of \(10^{-12}\) is reasonable.

Finally, restarting does not affect the rate of convergence for the double-precision implementation.
This matches the theory in \cref{sec:theory:ir} that the restarted iteration is equivalent to the non-restarted iteration when round-off errors are negligible.

We repeat the experiment with block Gauss-Seidel, as shown in \Cref{fig:convergence:gs-cage14,fig:convergence:gs-language}.
\begin{figure}
	\centering
	\begin{tikzpicture}
\begin{semilogyaxis}[
      width=4in, height=2.7in,
      ylabel={\(\infty\)-norm forward error},
      ytickten={4, 2, 0, -2, -4, -6, -8, -10, -12, -14, -16},
      xlabel=Iterations, xmin=0, xmax=50,
      xticklabel style={/pgf/number format/fixed, /pgfplots/scaled x ticks=false},
      minor xtick={1,6,11,16,21,26,31,36,41,46},
      grid=minor, grid style={line width=1pt, draw=gray!50},
      legend style={at={(0.5,1.02)}, anchor=south, font=\small},
      legend cell align=left, legend columns=2]
\addplot[dashed,color_bfloat16,very thick] coordinates {
    (0, 1) (1, 1.320591e+00) (2, 1.273330e+00) (3, 3.273540e-01) (4, 4.435058e-02) (5, 1.052346e-02) (6, 3.700567e-03) (7, 3.058340e-03) (8, 3.069183e-03) (9, 3.059522e-03) (10, 3.057608e-03) (11, 3.057345e-03) (12, 3.057253e-03) (13, 3.057226e-03) (14, 3.057219e-03) (15, 3.057216e-03) (16, 3.057216e-03) (17, 3.057215e-03) (18, 3.057215e-03) (19, 3.057215e-03) (20, 3.057215e-03) (21, 3.057215e-03) (22, 3.057215e-03) (23, 3.057215e-03) (24, 3.057215e-03) (25, 3.057215e-03) (26, 3.057215e-03) (27, 3.057215e-03) (28, 3.057215e-03) (29, 3.057215e-03) (30, 3.057215e-03) (31, 3.057215e-03) (32, 3.057215e-03) (33, 3.057215e-03) (34, 3.057215e-03) (35, 3.057215e-03) (36, 3.057215e-03) (37, 3.057215e-03) (38, 3.057215e-03) (39, 3.057215e-03) (40, 3.057215e-03) (41, 3.057215e-03) (42, 3.057215e-03) (43, 3.057215e-03) (44, 3.057215e-03) (45, 3.057215e-03) (46, 3.057215e-03) (47, 3.057215e-03) (48, 3.057215e-03) (49, 3.057215e-03) (50, 3.057215e-03) 
};
\addlegendentry[align=left]{ bfloat16 - no restarts };
\addplot[solid,color_bfloat16,very thick] coordinates {
    (0, 1) (1, 1.320591e+00) (2, 1.273330e+00) (3, 3.265927e-01) (4, 4.462016e-02) (5, 1.050109e-02) (6, 3.771033e-03) (7, 3.443561e-03) (8, 1.187584e-03) (9, 3.855802e-04) (10, 1.233103e-04) (11, 3.881725e-05) (12, 3.604872e-05) (13, 1.155257e-05) (14, 3.724426e-06) (15, 1.217927e-06) (16, 3.777648e-07) (17, 3.551374e-07) (18, 1.130594e-07) (19, 3.646756e-08) (20, 1.192347e-08) (21, 3.692561e-09) (22, 3.477059e-09) (23, 1.106502e-09) (24, 3.563600e-10) (25, 1.165153e-10) (26, 3.608162e-11) (27, 3.395067e-11) (28, 1.079760e-11) (29, 3.477866e-12) (30, 1.137057e-12) (31, 3.520387e-13) (32, 3.313261e-13) (33, 1.054099e-13) (34, 3.390923e-14) (35, 1.104400e-14) (36, 3.407318e-15) (37, 3.195170e-15) (38, 1.887379e-15) (39, 1.887379e-15) (40, 1.887379e-15) (41, 1.887379e-15) (42, 1.887379e-15) (43, 1.887379e-15) (44, 1.887379e-15) (45, 1.887379e-15) (46, 1.887379e-15) (47, 1.887379e-15) (48, 1.887379e-15) (49, 1.887379e-15) (50, 1.887379e-15) 
};
\addlegendentry[align=left]{ bfloat16 - restarts };
\addplot[dashed,color_half,very thick] coordinates {
    (0, 1) (1, 1.320591e+00) (2, 1.272153e+00) (3, 3.274225e-01) (4, 4.466155e-02) (5, 1.016958e-02) (6, 3.843965e-03) (7, 1.367643e-03) (8, 4.582280e-04) (9, 3.870485e-04) (10, 3.855929e-04) (11, 3.852117e-04) (12, 3.851052e-04) (13, 3.850721e-04) (14, 3.850619e-04) (15, 3.850589e-04) (16, 3.850579e-04) (17, 3.850576e-04) (18, 3.850575e-04) (19, 3.850574e-04) (20, 3.850574e-04) (21, 3.850574e-04) (22, 3.850574e-04) (23, 3.850574e-04) (24, 3.850574e-04) (25, 3.850574e-04) (26, 3.850574e-04) (27, 3.850574e-04) (28, 3.850574e-04) (29, 3.850574e-04) (30, 3.850574e-04) (31, 3.850574e-04) (32, 3.850574e-04) (33, 3.850574e-04) (34, 3.850574e-04) (35, 3.850574e-04) (36, 3.850574e-04) (37, 3.850574e-04) (38, 3.850574e-04) (39, 3.850574e-04) (40, 3.850574e-04) (41, 3.850574e-04) (42, 3.850574e-04) (43, 3.850574e-04) (44, 3.850574e-04) (45, 3.850574e-04) (46, 3.850574e-04) (47, 3.850574e-04) (48, 3.850574e-04) (49, 3.850574e-04) (50, 3.850574e-04) 
};
\addlegendentry[align=left]{ half - no restarts };
\addplot[solid,color_half,very thick] coordinates {
    (0, 1) (1, 1.320591e+00) (2, 1.272153e+00) (3, 3.251016e-01) (4, 4.478396e-02) (5, 1.027602e-02) (6, 3.893644e-03) (7, 3.516516e-03) (8, 1.214073e-03) (9, 3.923475e-04) (10, 1.242425e-04) (11, 4.122724e-05) (12, 3.693070e-05) (13, 1.188056e-05) (14, 3.830377e-06) (15, 1.239461e-06) (16, 4.072780e-07) (17, 3.654540e-07) (18, 1.167872e-07) (19, 3.752658e-08) (20, 1.216119e-08) (21, 3.994336e-09) (22, 3.584650e-09) (23, 1.144530e-09) (24, 3.681229e-10) (25, 1.195461e-10) (26, 3.918959e-11) (27, 3.518637e-11) (28, 1.122808e-11) (29, 3.611104e-12) (30, 1.168109e-12) (31, 3.836767e-13) (32, 3.443154e-13) (33, 1.099314e-13) (34, 3.536195e-14) (35, 1.147649e-14) (36, 3.764346e-15) (37, 3.373955e-15) (38, 1.776357e-15) (39, 1.776357e-15) (40, 1.776357e-15) (41, 1.998401e-15) (42, 1.998401e-15) (43, 1.998401e-15) (44, 1.998401e-15) (45, 1.998401e-15) (46, 1.998401e-15) (47, 1.998401e-15) (48, 1.998401e-15) (49, 1.998401e-15) (50, 1.998401e-15) 
};
\addlegendentry[align=left]{ half - restarts };
\addplot[dashed,color_single,very thick] coordinates {
    (0, 1) (1, 1.320591e+00) (2, 1.272352e+00) (3, 3.237934e-01) (4, 4.446093e-02) (5, 1.012051e-02) (6, 3.801549e-03) (7, 1.316061e-03) (8, 4.087395e-04) (9, 1.335524e-04) (10, 4.422423e-05) (11, 1.438124e-05) (12, 4.692236e-06) (13, 1.536190e-06) (14, 5.024474e-07) (15, 1.647887e-07) (16, 5.454630e-08) (17, 4.602019e-08) (18, 4.586700e-08) (19, 4.582390e-08) (20, 4.581111e-08) (21, 4.580733e-08) (22, 4.580617e-08) (23, 4.580580e-08) (24, 4.580569e-08) (25, 4.580565e-08) (26, 4.580564e-08) (27, 4.580564e-08) (28, 4.580564e-08) (29, 4.580564e-08) (30, 4.580564e-08) (31, 4.580564e-08) (32, 4.580564e-08) (33, 4.580564e-08) (34, 4.580564e-08) (35, 4.580564e-08) (36, 4.580564e-08) (37, 4.580564e-08) (38, 4.580564e-08) (39, 4.580564e-08) (40, 4.580564e-08) (41, 4.580564e-08) (42, 4.580564e-08) (43, 4.580564e-08) (44, 4.580564e-08) (45, 4.580564e-08) (46, 4.580564e-08) (47, 4.580564e-08) (48, 4.580564e-08) (49, 4.580564e-08) (50, 4.580564e-08) 
};
\addlegendentry[align=left]{ single - no restarts };
\addplot[solid,color_single,very thick] coordinates {
    (0, 1) (1, 1.320591e+00) (2, 1.272352e+00) (3, 3.313353e-01) (4, 4.557322e-02) (5, 1.039874e-02) (6, 3.888319e-03) (7, 3.537105e-03) (8, 1.217100e-03) (9, 3.935594e-04) (10, 1.247805e-04) (11, 4.106452e-05) (12, 3.704460e-05) (13, 1.191153e-05) (14, 3.835581e-06) (15, 1.245365e-06) (16, 4.053010e-07) (17, 3.668220e-07) (18, 1.171651e-07) (19, 3.761495e-08) (20, 1.222489e-08) (21, 3.982430e-09) (22, 3.600186e-09) (23, 1.148698e-09) (24, 3.686623e-10) (25, 1.197605e-10) (26, 3.897181e-11) (27, 3.522295e-11) (28, 1.123245e-11) (29, 3.607125e-12) (30, 1.170875e-12) (31, 3.810720e-13) (32, 3.447016e-13) (33, 1.099423e-13) (34, 3.534884e-14) (35, 1.144516e-14) (36, 3.691376e-15) (37, 3.334510e-15) (38, 1.776357e-15) (39, 1.776357e-15) (40, 1.776357e-15) (41, 1.998401e-15) (42, 1.998401e-15) (43, 1.998401e-15) (44, 1.998401e-15) (45, 1.998401e-15) (46, 1.998401e-15) (47, 1.998401e-15) (48, 1.998401e-15) (49, 1.998401e-15) (50, 1.998401e-15) 
};
\addlegendentry[align=left]{ single - restarts };
\addplot[dashed,color_double,very thick] coordinates {
    (0, 1) (1, 1.320591e+00) (2, 1.272352e+00) (3, 3.239457e-01) (4, 4.470530e-02) (5, 1.013700e-02) (6, 3.804995e-03) (7, 1.315875e-03) (8, 4.102703e-04) (9, 1.342864e-04) (10, 4.441246e-05) (11, 1.444222e-05) (12, 4.695177e-06) (13, 1.535417e-06) (14, 5.011003e-07) (15, 1.635277e-07) (16, 5.337886e-08) (17, 1.742411e-08) (18, 5.683640e-09) (19, 1.853595e-09) (20, 6.050895e-10) (21, 1.974830e-10) (22, 6.438606e-11) (23, 2.098990e-11) (24, 6.843283e-12) (25, 2.231094e-12) (26, 7.271951e-13) (27, 2.368430e-13) (28, 7.725237e-14) (29, 2.518855e-14) (30, 8.196936e-15) (31, 2.640183e-15) (32, 1.887379e-15) (33, 1.887379e-15) (34, 1.887379e-15) (35, 1.887379e-15) (36, 1.887379e-15) (37, 1.887379e-15) (38, 1.887379e-15) (39, 1.887379e-15) (40, 1.887379e-15) (41, 1.887379e-15) (42, 1.887379e-15) (43, 1.887379e-15) (44, 1.887379e-15) (45, 1.887379e-15) (46, 1.887379e-15) (47, 1.887379e-15) (48, 1.887379e-15) (49, 1.887379e-15) (50, 1.887379e-15) 
};
\addlegendentry[align=left]{ double - no restarts };
\addplot[solid,color_double,very thick] coordinates {
    (0, 1) (1, 1.320591e+00) (2, 1.272352e+00) (3, 3.250966e-01) (4, 4.467458e-02) (5, 1.022813e-02) (6, 3.829881e-03) (7, 3.474227e-03) (8, 1.199032e-03) (9, 3.876344e-04) (10, 1.229937e-04) (11, 4.049113e-05) (12, 3.654272e-05) (13, 1.175470e-05) (14, 3.784994e-06) (15, 1.226313e-06) (16, 3.995530e-07) (17, 3.613051e-07) (18, 1.153930e-07) (19, 3.710734e-08) (20, 1.204956e-08) (21, 3.921290e-09) (22, 3.544294e-09) (23, 1.130733e-09) (24, 3.637531e-10) (25, 1.181072e-10) (26, 3.842666e-11) (27, 3.473942e-11) (28, 1.108224e-11) (29, 3.565262e-12) (30, 1.158064e-12) (31, 3.766813e-13) (32, 3.405941e-13) (33, 1.086007e-13) (34, 3.491976e-14) (35, 1.132977e-14) (36, 3.683111e-15) (37, 3.313022e-15) (38, 1.887379e-15) (39, 1.887379e-15) (40, 1.887379e-15) (41, 1.887379e-15) (42, 1.887379e-15) (43, 1.887379e-15) (44, 1.887379e-15) (45, 1.887379e-15) (46, 1.998401e-15) (47, 1.998401e-15) (48, 1.998401e-15) (49, 1.998401e-15) (50, 1.998401e-15) 
};
\addlegendentry[align=left]{ double - restarts };
\end{semilogyaxis}
\end{tikzpicture}
	\caption{Convergence rate of block Gauss-Seidel on \matrixname{cage14} with different precisions.  Vertical lines indicate the iterations in which restarts occur.}
	\label{fig:convergence:gs-cage14}
\end{figure}
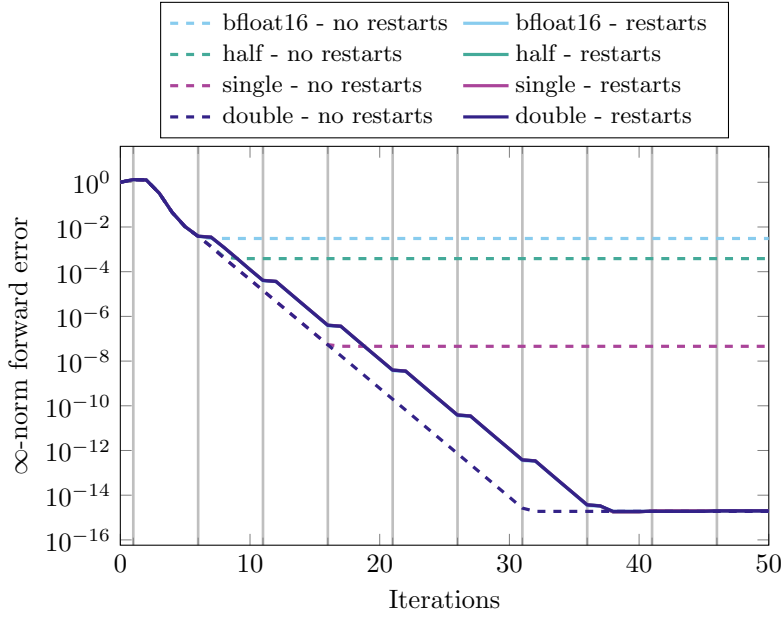
\begin{figure}
	\centering
	\input{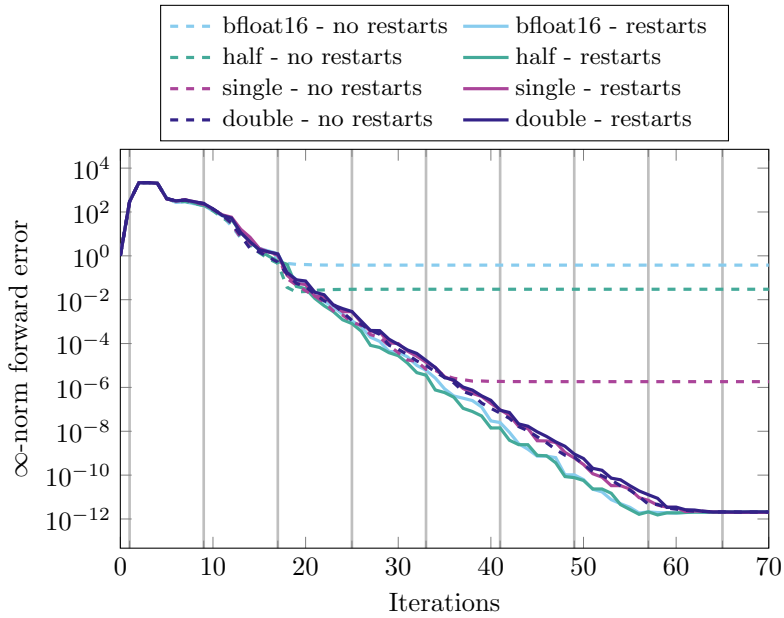}
	\caption{Convergence rate of block Gauss-Seidel on \matrixname{language} with different precisions.  Vertical lines indicate the iterations in which restarts occur.}
	\label{fig:convergence:gs-language}
\end{figure}
Like Jacobi, mixed-precision Gauss-Seidel converged at similar rates as double precision and, when restarted, could achieve full double-precision accuracy.

For \matrixname{cage14}, restarting Gauss-Seidel had a notable effect on the rate of convergence.
This is caused by the restart iterations performing a Jacobi step, which has slower convergence, instead of a Gauss-Seidel step.
Note that such a difference was not obvious for the \matrixname{language} matrix, although the restarted configurations do appear to generally be slightly slower.
In part, Jacobi converged nearly as quickly as Gauss-Seidel for this matrix (compare \cref{fig:convergence:jac-language,fig:convergence:gs-language}),
so the restart iterations had almost as much improvement as the normal iterations.
Additionally, the iteration is sensitive to small perturbations in the working solution; for example, note that the single and double precision lines slightly separate well before reaching single-precision accuracy.
This instability helps hide the difference in convergence between the restarted and non-restarted configuration.
Note that the initial transient growth of the error suggests large pseudospectral radius~\cite[(18.8)]{higham2002accuracy}, which is likely the source of the instability.

Now, we consider the effect on performance.
The iterations are run until they compute a solution that satisfies
\[
	\frac{\norm{\infty}{x - x^*}}{\norm{\infty}{x^*}} \leq 10^{-10}.
\]
Note that one of the primary benefits of stationary iterations over Krylov-subspace methods on modern systems is the lack of global reductions, so in many production use cases, a fixed iteration count is used for the stationary iteration instead of a runtime convergence criteria.
(Furthermore, the forward error is unknown in real applications.)
We mirror this by using a fixed number of iterations for the timed solve, which is computed by a preliminary, untimed solve.
Additionally, we restart the mixed-precision configurations based on the number of iterations required for the double precision implementation.
The BFloat16 restart length is a fifth of the total double-precision iterations, the half-precision restart length is a fourth, and the single-precision restart is a half, all rounding down.
These restart heuristics were chosen based on \cref{fig:convergence:jac-cage14,fig:convergence:jac-language,fig:convergence:gs-cage14,fig:convergence:gs-language}.

\sisetup{separate-uncertainty=true,table-align-uncertainty=true}
\setlength{\tabcolsep}{4.5pt}
\begin{sidewaystable}
	\caption{Performance of the mixed-precision approach for Jacobi iterations.  Uncertainty intervals indicate 3 standard deviations.}
	\label{tab:performance-jacobi}
	\centering
	\begin{tabular}{l*{4}{S[table-format=5.2(5),table-auto-round]S[table-format=6.0]}}
		& \multicolumn{2}{c}{Double} & \multicolumn{2}{c}{Single} & \multicolumn{2}{c}{Half} & \multicolumn{2}{c}{BFloat16} \\
\cmidrule(lr){2-3}\cmidrule(lr){4-5}\cmidrule(lr){6-7}\cmidrule(lr){8-9}
Matrix & {Time (ms)} & {Iter.} & {Time (ms)} & {Iter.} & {Time (ms)} & {Iter.} & {Time (ms)} & {Iter.} \\
\midrule
\matrixname{atmosmodj} & 4263.46 +- 20.89 & 44138 & 3793.24 +- 14.6 & 44570 & 8569.95 +- 52.86 & 114487 & 14627.7 +- 57.38 & 194997 \\
\matrixname{cage14} & 98.05 +- 0.64 & 181 & 71.06 +- 2.0 & 181 & 65.63 +- 0.09 & 181 & 65.89 +- 0.08 & 181 \\
\matrixname{language} & 132.57 +- 0.38 & 64 & 116.44 +- 0.06 & 65 & 70.46 +- 0.32 & 61 & 79.14 +- 0.42 & 66 \\
\matrixname{thermomech\_dM} & 2386.03 +- 1.51 & 87372 & 1711.34 +- 4.9 & 87373 & 1197.79 +- 3.02 & 87265 & 1207.46 +- 0.44 & 85700 \\
\matrixname{torso2} & 1.26 +- 0.05 & 236 & 1.21 +- 0.01 & 236 & 1.08 +- 0.01 & 236 & 1.11 +- 0.08 & 239 \\

	\end{tabular}
\end{sidewaystable}%
\begin{sidewaystable}
	\caption{Performance of the mixed-precision approach for Gauss-Seidel iterations.  Uncertainty intervals indicate 3 standard deviations.}
	\label{tab:performance-gs}
	\centering
	\begin{tabular}{l*{4}{S[table-format=5.2(5),table-auto-round]S[table-format=6.0]}}
		& \multicolumn{2}{c}{Double} & \multicolumn{2}{c}{Single} & \multicolumn{2}{c}{Half} & \multicolumn{2}{c}{BFloat16} \\
\cmidrule(lr){2-3}\cmidrule(lr){4-5}\cmidrule(lr){6-7}\cmidrule(lr){8-9}
Matrix & {Time (ms)} & {Iter.} & {Time (ms)} & {Iter.} & {Time (ms)} & {Iter.} & {Time (ms)} & {Iter.} \\
\midrule
\matrixname{Dubcova2} & 1010.86 +- 0.24 & 30141 & 1022.42 +- 0.16 & 30142 & 977.87 +- 0.1 & 29946 & 1000.98 +- 0.11 & 30144 \\
\matrixname{atmosmodj} & 2706.6 +- 2.63 & 23394 & 2407.73 +- 1.6 & 23602 & 5977.33 +- 5.26 & 66476 & 9299.84 +- 3.68 & 103289 \\
\matrixname{cage14} & 14.67 +- 0.11 & 22 & 11.81 +- 0.1 & 23 & 12.33 +- 0.13 & 25 & 12.9 +- 0.1 & 26 \\
\matrixname{language} & 121.44 +- 0.38 & 53 & 97.89 +- 0.44 & 51 & 64.53 +- 0.45 & 49 & 68.62 +- 0.48 & 51 \\
\matrixname{poisson3Db} & 613.11 +- 1.01 & 11169 & 601.73 +- 0.25 & 11161 & 588.1 +- 0.15 & 11288 & 819.61 +- 0.15 & 15726 \\
\matrixname{thermomech\_dM} & 1.77 +- 0.02 & 36 & 1.41 +- 0.03 & 36 & 1.23 +- 0.02 & 39 & 1.26 +- 0.01 & 39 \\
\matrixname{torso2} & 9.43 +- 0.09 & 460 & 9.15 +- 0.09 & 453 & 8.63 +- 0.1 & 449 & 8.17 +- 0.1 & 422 \\

	\end{tabular}
\end{sidewaystable}%
\Cref{tab:performance-jacobi,tab:performance-gs} show the performance of block Jacobi and block Gauss-Seidel, respectively.
For all of the tested matrices, if double-precision converged, the iteration also converged with mixed precision.
Furthermore, in almost all cases, reducing precision improved the overall performance.
(The meaningful exceptions are BFloat16 Gauss-Seidel with \matrixname{poisson3Db} and all four combinations of iterations and 16-bit precisions with \matrixname{atmosmodj}.)
Taking the geometric mean of the per-matrix speedups, the Jacobi iteration is \(1.21\times\) faster with single-precision, \(1.27\times\) faster with half-precision, and \(1.10\times\) faster with BFloat16 compared to double precision;
the Gauss-Seidel iteration is instead \(1.12\times\) faster, \(1.08\times\) faster, and \(0.95\times\) slower, respectively.
If we exclude cases where slow-downs occur, the speedups for Jacobi increase to \(1.60\times\) and \(1.54\times\) for half-precision and BFloat16, respectively, and the speedups for Gauss-Seidel increase to \(1.25\times\) and \(1.27\times\), respectively.

There are several factors that allow mixed precision to achieve a higher speedup with Jacobi than Gauss-Seidel.
First, Gauss-Seidel often converges more quickly than Jacobi, so a higher percent of the iterations will be double-precision residuals compared to Jacobi.
Additionally, recall that each row in a Jacobi iteration is independent of all others, this allows each iteration to be handled in a single, parallel CUDA kernel.
Gauss-Seidel, on the other hand, has dependencies between rows based on a coloring, thus each color is launched with a separate CUDA kernel.
Each CUDA kernel launched in an iteration introduces some overhead for launching the kernel and the reduced utilization as the kernel finishes.
Furthermore, some colors in the Gauss-Seidel may have a small number of rows, resulting in poor utilization of the GPU, while Jacobi only has a single, high-occupancy kernel.

Note that, when slowdowns occur, they are driven by an increased number of iterations;
for all of the tests on the matrix, reducing the precision always reduces the average time per-iteration.
(The one exception is the \matrixname{Dubcova2} matrix with single-precision, which has a \SI{1}{\%} slowdown.)
To study this difference in iteration count, \cref{fig:perf-conv:jac-atmosmodj} plots the convergence of the Jacobi iteration with different precisions on the \matrixname{atmosmodj} matrix.
\begin{figure}
	\centering
	\input{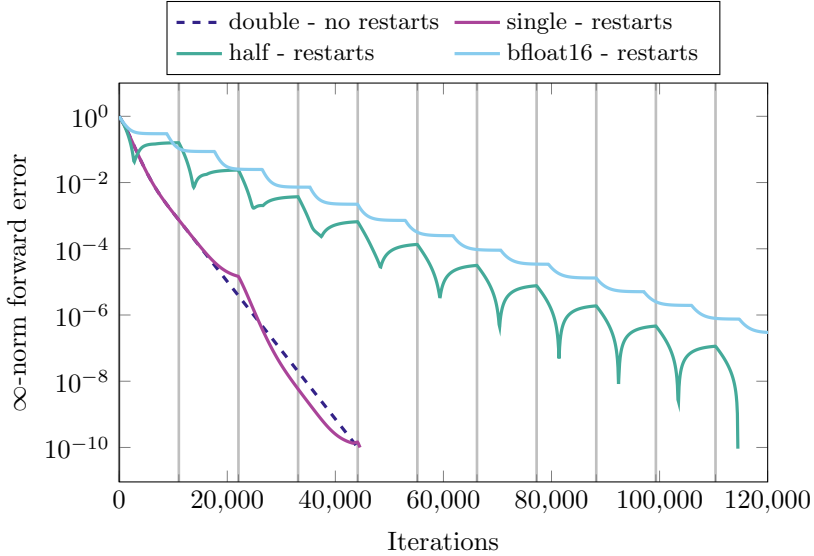}
	\caption{Convergence rate of block Jacobi on \matrixname{atmosmodj} with different precisions.  Vertical lines indicate the iterations in which restarts occur for the half-precision configuration.}
	\label{fig:perf-conv:jac-atmosmodj}
\end{figure}
Interestingly, the half-precision implementation forms a scallop-like convergence pattern, with the peaks occurring at the restarts.
The convergence of BFloat16, on the other hand, just stagnates after the error improves by a fraction of a digit.
Thus, for both of these cases, increasing the restart frequency is expected to reduce the iteration count and improve performance for the 16-bit precisions.
Note that for Jacobi, the iteration matrix has a spectral radius\footnotemark{} of approximately \num{0.999521} and an \(\infty\)-norm of \num{0.99999985}.
\footnotetext{See \texttt{analysis.jl} in our code artifact for more details.}
Thus, \(\norm{\infty}{A^k}\) is always less than 1 and goes to zero as \(k\) goes to infinity in exact or double precision.
However, for half-precision, unit round-off is comparable to the gap between the spectral radius and 1 and is significantly larger than the gap between the norm and 1; and for BFloat16, the unit round-off is larger than both gaps.
Hence, rounding errors may perturb the effective spectral radius and norm to be larger than 1, which can delay or even prevent convergence.
The behavior of Gauss-Seidel on \matrixname{atmosmodj} and \matrixname{poisson3Db} appears to be similar.


\subsection{OVERFLOW}

We start by testing the effect of half-precision on OVERFLOW's SSOR linear solver~\cite{jackson2024gpu}.
OVERFLOW is a compressible CFD code developed by NASA and is based on structured overset grids.
Thus, the overarching paradigm is a Newton iteration for nonlinear equations, with a linear solve on each grid for each Newton iteration.

OVERFLOW provides many linear solver options, but we focused on the multi-colored, SSOR solver.
The linear system for each grid comes from discretizing the mesh with a 7-point finite-difference stencil.
Because it is placed within a Newton iteration, only a few digits of accuracy are needed from the linear solver~\cite{derlaga2020recent}.
OVERFLOW already converts the linear system to single-precision before running the SSOR.
So, we test the effect of reducing the linear system to BFloat16, including both the matrix and vector elements.

We tested OVERFLOW with a case from the 6th drag prediction workshop.
Specifically, we used the \verb|WB_ae2p50_medium| case, which is a steady-state simulation from an angle-of-attack sweep from aero-elastic deflections.
Flow conditions are: \(M = 0.85\), \(Re = 5 \text{ million}\), Reference temperature = 100\textdegree F, and Angle-of-attack = 2.5\textdegree.
The solver was run for \num{50000} Newton iterations with ten SSOR iterations per Newton step.

\Cref{fig:overflow-convergence} shows the convergence of the baseline, half-precision, and BFloat16 configurations on the overall mesh.
\begin{figure}
	\centering
	\input{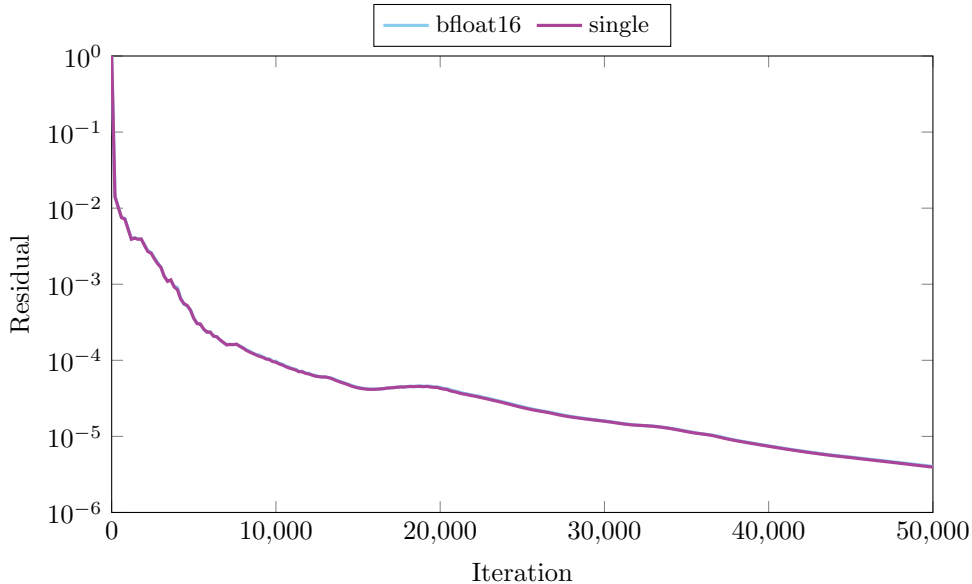}
	\caption{Effect of mixed-precision on convergence for OVERFLOW.  (Note that the convergence lines overlap almost identically.)}
	\label{fig:overflow-convergence}
\end{figure}%
The plotted residual is the sum of the per-grid residuals.
Note that the nonlinear solve converged almost identically between the configurations, and the achievable accuracy was not limited by the reduction in precision.
Heuristically, this is to be expected: OVERFLOW uses a Newton iteration with many steps, so accuracy of individual linear solves is not critical.
While the configurations achieved similar convergence, the low-precision configuration had a \SI{13.7}{\%} speedup for the overall application.
(This speedup includes the setup costs, although the Newton iteration takes \SI{99.9}{\%} of the runtime for both configurations.)

To provide a more detailed performance breakdown, we profiled the code with NVIDIA Nsight Systems on 500 Newton iterations.
\Cref{fig:overflow-breakdown} shows the time spent in different parts of the Newton solve.
(Note that, because the setup phase is negligible for the actual simulation, its time is omitted.)
\begin{figure}
	\centering
	\begin{tikzpicture}
		\begin{axis}[
			ybar stacked,
			bar width=40pt,
			enlarge x limits=0.5,
			legend style={at={(1.05,0.95)}, anchor=north west},
			legend cell align={left}, reverse legend,
			ylabel={Time (seconds)},
			ymin=0, ymax=150, ytick distance=30,
			symbolic x coords={bfloat,single},
			xtick=data, xticklabels={bfloat16, float},
			]
			\addplot[ybar,fill=color_single] plot coordinates {(bfloat,30.718) (single,30.754)}; 
			\addlegendentry{Other}
			\addplot[ybar,fill=color_half] plot coordinates {(bfloat,42.849) (single,44.347)}; 
			\addlegendentry{Build linear system}
			\addplot[ybar,fill=color_bfloat16] plot coordinates {(bfloat,37.758) (single,51.841)}; 
			\addlegendentry{SSOR kernel}
		\end{axis}
	\end{tikzpicture}
	\caption{Performance breakdown for OVERFLOW on 500 Newton iterations.}
	\label{fig:overflow-breakdown}
\end{figure}
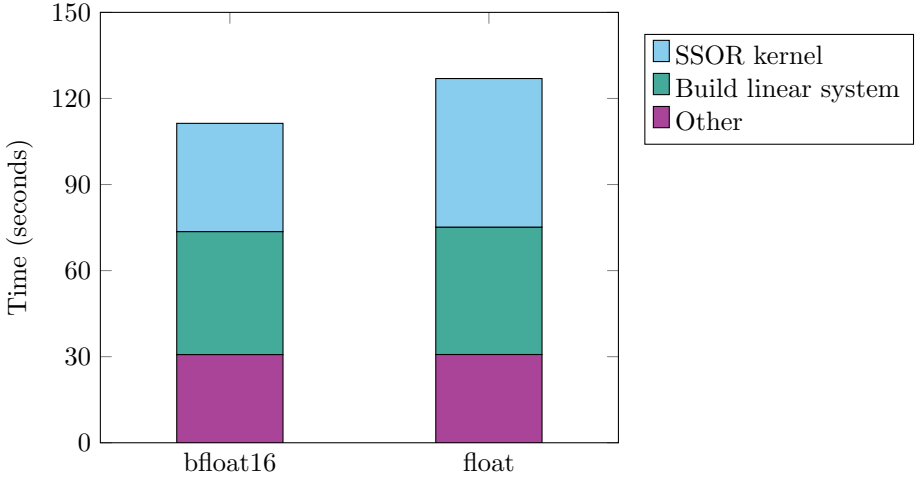
Thus, for the above converged run, that portion will become negligible.
The SSOR kernel itself saw a \(1.37\times\) speedup.
This corresponds to a halving of the floating point data but not reducing the integer data or reducing the computational overhead.
The performance of the rest of the Newton solve was largely unaffected.

\section{Conclusions}

Using this mixed-precision approach for stationary iterative methods reduces the amount of data movement, and thus reduces the time to solution.
Theoretical analysis indicates that this approach will converge at a similar rate to full precision until the accuracy of the low-precision component is saturated, which was verified experimentally.
Furthermore, experimental results show that selectively using half-precision provided a \(1.60\times\) speedup for block Jacobi iterations and a \(1.25\times\) speedup for block Gauss-Seidel iterations.
Finally, we demonstrated that this technique can be used to accelerate OVERFLOW, a production fluid simulation code by \(1.14\times\).

Importantly, the effectiveness of this approach comes from the classic, mixed-precision iterative refinement paradigm, which has a strong history,
and the fact that the initial convergence is not affected by the reduction in precision.
This makes the timing of restarts important for performance.
However, our experimental results suggested that the ideal restart point for mixed-precision stationary methods cannot be simply determined for a given problem, which aligns with previous work on restarted iterative methods~\cite{lindquist2020improving,baker2009simple,haque2025gpu}.
Furthermore, the heuristic used in our experiments requires an estimate on the number of iterations needed for the full-precision implementation.
Thus, like previous work in mixed precision, our approach would benefit from the development of better restart strategies.

Besides compressible CFD, one of the most important modern uses of stationary iterations is as the smoother step for multigrid methods.
Note that the pre-smoother is run with an initial guess of 0.
So, based on the discussion in \cref{sec:theory}, this suggests that there will be little slowdown in the rate of convergence on the small number of smoother iterations used in practice.
Furthermore, by adding a full-precision residual computation after the prolongation step, \cref{sec:theory:ir} suggests mixed precision will also be effective for the post-smoother.
While there has been previous work on using mixed precision for multigrid methods~\cite{mccormick2021algebraic,tsai2024portable}, our approach may allow the mix of precisions to be more finely tuned, improving performance.

Finally, note that \cref{thm:main-result,thm:block-row-scaling} bound the accuracy given a normwise approximation error.
Thus, there is potential for using different compression schemes than just storing each element in a lower-precision floating-point format.
For example, storing blocks of the matrix in a block floating-point format or as a low-rank approximation may provide better trade-offs in performance and overall accuracy.

\section*{Acknowledgments}

Claude was used to generate and improve plotting scripts and to identify flaws and weaknesses in earlier drafts of the manuscript.
The authors assume responsibility for all content.

The authors would like to thank the NASA Subsonic Vehicle Technologies and Tools (SVTT) Project for supporting the implementation of this work in OVERFLOW. We would also like to thank the NASA Langley K-Cluster team for providing the compute resources to run preliminary OVERFLOW tests.

\bibliographystyle{siamplain}
\bibliography{mxpr-stationary}
\end{document}